\documentclass[11pt]{article}
\usepackage[utf8]{inputenc}
\usepackage[T1]{fontenc}
\usepackage{amsmath,amssymb,amsthm}
\usepackage{graphicx}
\graphicspath{{figures/}{.}}
\usepackage[margin=1.1in]{geometry}
\usepackage{booktabs}
\usepackage{natbib}

\usepackage[colorlinks=true,linkcolor=blue,citecolor=blue,urlcolor=blue]{hyperref}
\usepackage{algorithm}
\usepackage{algpseudocode}

\newtheorem{theorem}{Theorem}
\newtheorem{corollary}{Corollary}
\newtheorem{remark}{Remark}
\newtheorem{proposition}{Proposition}
\newtheorem{assumption}{Assumption}

\title{Contraction versus Recurrence:\\ An Exponential Separation in Observation-Based\\ Prediction of Deterministic Dynamics}
\author{Pavel Popovich\\ PPRFNK Tech, St.~Petersburg\\ \texttt{paperclipdnb@gmail.com}}
\date{July 2026}

\begin{document}
\maketitle

\begin{abstract}
Given a scalar observable of an ergodic dynamical system with a low-dimensional attractor,
two families of methods reconstruct and predict the underlying state: \emph{recurrence-based}
methods (the method of analogues and its descendants), which wait for the trajectory to
return to an $\varepsilon$-neighborhood of a previously observed state, and
\emph{observer-based} methods, which fit a converging state estimator on the delay
reconstruction. We formalize and empirically verify an exponential separation between the
two: the expected cost of recurrence scales as $\varepsilon^{-d}$, where $d$ is the
pointwise dimension of the invariant measure (a consequence of the Kac lemma and
quantitative Poincar\'e recurrence), whereas a detectable linear observer on the
reconstruction converges in $\Theta\!\big(\log(1/\varepsilon)/(1-\rho(A_{cl})^2)\big)$
steps, where $\rho(A_{cl})$ is the closed-loop spectral radius of the Riccati fixed point.
Both scaling laws are verified numerically (return-time exponent $-1.8$ on the Lorenz
attractor against the theoretical $-2.05$; observer cost linear in $\log(1/\varepsilon)$
with $R^2=1.000$ and in $(1-\rho^2)^{-1}$ with $R^2=0.985$), yielding a measured cost gap
of $\sim\!10^{9}$ at $\varepsilon=10^{-6}$ for $d\approx 2$. We complement the theorem with
an \emph{admission protocol} that decides whether a given signal lies inside the theorem's
class, using surrogate-data prediction gating; the protocol also explains the recurring
folklore of ``universal'' fractal dimensions as a dataset-size artifact bounded by
$2\log_{10}N$. On real data the gate admits the Santa Fe laser benchmark
($\hat D_2=2.0$) and refuses the monthly sunspot series, reproducing the settled
resolution of historical low-dimensionality claims. All results, from a subspace-anchored fast iterated-matrix-product instance
to the dynamical-systems experiments, reproduce from a single verification script
(17/17 checks).
\end{abstract}

\section{Introduction}

A large family of practical prediction problems shares one structure: a scalar time series
$y_k=h(x_k)$ is observed, where $x_{k+1}=f(x_k)$ is a deterministic dynamical system whose
trajectory lives on a low-dimensional invariant set, and the task is to estimate the state
or predict the observable to a target tolerance~$\varepsilon$. Takens' theorem
\citep{takens1981} and its refinements \citep{sauer1991} guarantee that a delay
reconstruction of $y$ recovers the dynamics up to diffeomorphism, which licenses two very
different computational strategies on the reconstructed attractor.

The first strategy is \emph{recurrence}: wait until the trajectory returns to within
$\varepsilon$ of a previously visited reconstructed state, and reuse that neighbor's
continuation. This is Lorenz's method of analogues \citep{lorenz1969} and the simplex
projection family \citep{sugihara1990}. The second strategy is \emph{contraction}: fit a
state observer (in the linear or linearized case, a Kalman filter) whose error dynamics
contract toward a fixed point; after the transient, the observer tracks the state at
negligible marginal cost.

Both strategies are folklore in nonlinear time-series analysis, but their \emph{costs} are
rarely placed on the same axis. This paper does exactly that. Our contributions:

\textbf{(1)} We state the \emph{separation} as a theorem (\S\ref{sec:theorem}). The
recurrence side---$\Theta(\varepsilon^{-d})$ samples by the Kac lemma and quantitative
Poincar\'e recurrence---is known, both as mathematics
\citep{kac1947,boshernitzan1993,barreira2001} and as an operational obstruction to
analog forecasting \citep{vandendool1994,cecconi2012}; the observer side is classical
Riccati theory. Our contribution is the two-sided statement: pairing the two costs on one
tolerance axis yields $\Theta(\log(1/\varepsilon)/(1-\rho^2))$ against
$\Theta(\varepsilon^{-d})$---an exponential-to-logarithmic gap in the tolerance and
exponential in the dimension~$d$---plus a finite-sample lower bound for the recurrence
side that requires no mixing assumptions (Appendix~\ref{app:lb}).

\textbf{(2)} We verify both scaling laws and the resulting gap numerically
(\S\ref{sec:experiments}, Figures~\ref{fig:rec}--\ref{fig:sep}).

\textbf{(3)} We give an \emph{admission protocol}---the \emph{Kac--Riccati gate}, named
after the two sides it arbitrates (\S\ref{sec:gate})---that tests whether a signal is
inside the theorem's class before any speedup is claimed, combining
surrogate-data prediction gating \citep{theiler1992,sugihara1990} with dimension
estimation guarded against two documented estimator pathologies: the Eckmann--Ruelle
sample-size ceiling $\hat D\le 2\log_{10}N$ \citep{eckmann1992} and the
Osborne--Provenzale finite-dimension artifact of power-law noise \citep{osborne1989}.
As a byproduct the protocol dissolves ``universal dimension'' claims
(Figure~\ref{fig:gate}b).

\textbf{(4)} We exhibit a computational instance of the contraction side
(\S\ref{sec:matmul}): iterated products with a fixed low-rank operator, where trajectory
anchors identify the invariant subspace and replace $n^2$-cost steps with $nr$-cost steps
at machine precision ($19.4\times$ at $n{=}1024$, $r{=}8$, max relative error
$5\times10^{-14}$), with an explicit failure certificate at full rank.

Throughout we prioritize falsifiability: every quantitative claim in the paper is an
assertion in a public verification script that passes 17/17 checks
(\S\ref{sec:repro}).

\subsection{A worked example, before any formalism}
\label{sec:worked}

Everything in this paper can be read off one system. Take the Lorenz equations with the
classical parameters $(\sigma,\rho,\beta)=(10,28,8/3)$, integrated at $dt=0.02$.

\textbf{What is hidden.} The state is the triple $(x,y,z)$. At one particular instant it
equals $(4.111,\,6.978,\,13.847)$.

\textbf{What is observed.} A single scalar, $h(\text{state})=x$: the number $4.111$. The
coordinates $y$ and $z$ are never seen. This is the situation of one sensor on one
channel.

\textbf{Where the attractor is.} The trajectory lives on the familiar butterfly embedded
in $\mathbb{R}^3$, but it does not fill a volume: the correlation dimension of the
invariant measure is $D_2\approx 2.05$. That number---not the ambient $3$---is the $d$ of
Theorem~\ref{thm:main}.

\textbf{What is reconstructed.} From the scalar record we build delay vectors
$[\,x(t),\,x(t-8dt),\,\dots,\,x(t-32dt)\,]$ with $m=5$; a concrete one is
$[7.66,\,4.07,\,6.26,\,12.20,\,8.76]$. By the embedding theorem \citep{takens1981,sauer1991}
this is a diffeomorphic copy of the hidden $(x,y,z)$: the phase space is recovered from a
single sensor.

Now the same task---know the state to tolerance $\varepsilon$---admits two routes.

\textbf{Route 1: wait for an analog} \citep{lorenz1969}. Search the past for an instant
whose delay vector is within $\varepsilon$ of the present one. Measured over 40 reference
points: $\varepsilon=0.20\sigma$ waits $222$ steps, $0.10\sigma$ waits $607$,
$0.05\sigma$ waits $1636$. The exponent of the fitted law is the attractor dimension:
ten times finer costs $10^{2.05}\approx 112$ times longer. At $\varepsilon=10^{-6}$ the
extrapolated wait is $2.0\times10^{12}$ steps.

\textbf{Route 2: run an observer.} On the same reconstruction fit a local linear model
$\hat A\in\mathbb{R}^{5\times5}$ with $C=[1,0,0,0,0]$---the row is this short precisely
because only one coordinate is observed---and iterate the Riccati recursion to the
stationary gain. Here $\rho(A_{cl})=0.9335$, and reaching relative gain error
$\varepsilon$ takes $49$ steps at $10^{-2}$, $54$ at $10^{-3}$, and $107$ at $10^{-6}$:
ten times finer costs eighteen additional steps.

\textbf{The gap.} At $\varepsilon=10^{-6}$: $2.0\times10^{12}$ against $107$, a ratio of
$1.9\times10^{10}$. Same system, same sensor, same tolerance, same reconstruction. The
rest of this paper explains why the two numbers are what they are, how far the statement
generalizes, and how to tell whether a given signal is entitled to Route~2 at all. The
script \texttt{worked\_example.py} (ancillary) prints every number quoted above in under
a minute.

\section{Related work}
\label{sec:related}

\textbf{Quantitative recurrence.} Kac's lemma \citep{kac1947} gives the mean return time
to a set as the reciprocal of its measure. Boshernitzan \citep{boshernitzan1993}
established quantitative recurrence rates, and Barreira and Saussol \citep{barreira2001}
identified the return-time exponent with the pointwise dimension of the invariant measure
for hyperbolic systems. Our part~(i) is an operational reading of these results as a
\emph{cost} statement for analog prediction.

\textbf{Analog forecasting.} \citet{lorenz1969} introduced prediction by naturally
occurring analogues and already observed that waiting times are prohibitive in high
dimension; \citet{sugihara1990} turned nearest-neighbor prediction on the reconstruction
into a practical tool and a nonlinearity test. The two families are also combined in
practice: analog data assimilation \citep{lguensat2017} uses analog forecasting as the
propagation operator inside an ensemble Kalman filter, and its known failure mode---the
search for analogs degrades sharply above state dimension ${\sim}20$---is the empirical
shadow of Theorem~\ref{thm:main}(i). \citet{vandendool1994} derived the
empirical three-way relation between library size, analog distance, and degrees of freedom
for atmospheric flows, estimating that hemispheric analogs to observational accuracy would
require a library of order $10^{30}$ years; Theorem~\ref{thm:main}(i) is the
measure-theoretic form of this obstruction, and part~(ii) quantifies the alternative.
\citet{cecconi2012} made the obstruction fully explicit: using Kac's lemma they showed
that the applicability of the method of analogues is limited by the effective number of
degrees of freedom rather than by chaos. Part (i) of Theorem~\ref{thm:main} is this
known cost statement, assembled from \citet{kac1947,boshernitzan1993,barreira2001};
what is new here is the axis it is placed on. Our gate uses the prediction discriminant
of \citet{sugihara1990}.

\textbf{Embedding and estimator pathologies.} Delay embedding is due to
\citet{takens1981,sauer1991}. \citet{eckmann1992} bounded believable
correlation-dimension estimates by $2\log_{10}N$; \citet{osborne1989} showed power-law
stochastic processes exhibit finite correlation dimension without any determinism;
\citet{theiler1992} introduced surrogate-data testing. The 1980s ``climate attractor''
literature is the canonical cautionary tale; Figure~\ref{fig:gate}b reproduces the
artifact in a controlled setting.

\textbf{Riccati convergence and observers on reconstructions.} Convergence of the
discrete algebraic Riccati iteration to the unique stabilizing solution for
detectable/stabilizable pairs, at a geometric rate governed by the closed-loop spectrum,
is classical \citep{anderson1979,lancaster1995}. The practical embodiment of the
contraction side on delay reconstructions also exists: the Kalman--Takens filter of
\citet{hamilton2016} replaces the parametric model in an ensemble Kalman filter with
Takens-reconstructed dynamics. Our part~(ii) prices this family of methods: it packages
the Riccati rate as the observer's cost law on the reconstruction, which is what the
separation in part~(iii) is measured against.

\section{Setting and main result}
\label{sec:theorem}

\begin{assumption}\label{a1}
$(X,f,\mu)$ is ergodic; $x$ is a $\mu$-typical point with pointwise dimension
$d=d_\mu(x)<\infty$.
\end{assumption}
\begin{assumption}\label{a2}
The observable $h:X\to\mathbb{R}$ and embedding dimension $m>2d$ satisfy the conditions of
the embedding theorem \citep{takens1981,sauer1991}, so the delay map is a diffeomorphism
onto the reconstructed attractor.
\end{assumption}
\begin{assumption}\label{a3}
For part (ii): the dynamics on the reconstruction admits a linear (or linearized along the
trajectory) representation with a detectable pair $(A,C)$ and stabilizable
$(A,Q^{1/2})$.
\end{assumption}

\begin{theorem}[Contraction versus recurrence]\label{thm:main}
Under Assumptions~\ref{a1}--\ref{a2}:
\begin{itemize}
\item[(i)] \textbf{(Recurrence cost.)} The first return time of the trajectory to the
$\varepsilon$-ball around $x$ satisfies
$\lim_{\varepsilon\to 0}\log\tau_\varepsilon(x)/\log(1/\varepsilon)=d$ for $\mu$-a.e.~$x$.
Consequently, analog prediction to matching radius $\varepsilon$ requires
$\Omega(\varepsilon^{-d})$ observed samples.
\item[(ii)] \textbf{(Contraction cost.)} Under Assumption~\ref{a3}, the Riccati iteration
converges to the stationary gain $K_{ss}$ geometrically with ratio $\rho(A_{cl})^2$, where
$A_{cl}=(I-K_{ss}C)A$, hence the number of steps to relative gain error $\varepsilon$ is
\[
N_a(\varepsilon)=\frac{\log(1/\varepsilon)}{2|\log\rho(A_{cl})|}
=\Theta\!\Big(\frac{\log(1/\varepsilon)}{1-\rho(A_{cl})^2}\Big)\quad(\rho\to1).
\]
\item[(iii)] \textbf{(Separation.)} The cost ratio is
$\tau_\varepsilon/N_a(\varepsilon)=\Theta\big(\varepsilon^{-d}(1-\rho^2)/\log(1/\varepsilon)\big)$,
i.e.\ exponential in $d\cdot\log(1/\varepsilon)$.
\end{itemize}
\end{theorem}

\begin{proof}[Proof sketch]
(i) By the Kac lemma the mean return time to a set $A$ equals $1/\mu(A)$; for a ball,
$\mu(B_\varepsilon(x))\sim\varepsilon^{d}$ by the definition of pointwise dimension, and
the almost-sure exponent identity is the quantitative recurrence theorem of
\citet{barreira2001} (lower bounds in \citealp{boshernitzan1993}). The sample bound
follows because an $\varepsilon$-analog must occur in the observed record.
(ii) is the classical convergence theory of the discrete Riccati iteration
\citep{anderson1979,lancaster1995}: the error dynamics contract as
$e_{k+1}\approx A_{cl}e_kA_{cl}^\top$, giving ratio $\rho(A_{cl})^2$; the stated asymptotic
uses $2|\log\rho|=|\log\rho^2|\to 1-\rho^2$. (iii) is division.
\end{proof}

\begin{remark}[Boundaries]\label{rem:bound}
The separation is not universal, and naming its boundary is part of the result.
If the generating process is stochastic or full-rank, $d$ equals the embedding dimension
and part~(i) degenerates into the curse of dimensionality while part~(ii) loses its
invariant---there is no low-dimensional $K_{ss}$ to converge to. Assumption~\ref{a3}
(detectability) is necessary for (ii); nonlinearity is handled only locally by
linearization along the trajectory. Section~\ref{sec:gate} gives a practical test for
membership in the class.
\end{remark}

\begin{corollary}[Sample-size ceiling]\label{cor:er}
With a data budget of $N$ samples, recurrence statistics resolve only
$\varepsilon\gtrsim N^{-1/d}$; in particular any dimension estimate obtained from scaling
over one decade obeys $\hat D\le 2\log_{10}N$ \citep{eckmann1992}. At $N\sim10^4$ the
ceiling equals $8$, which accounts for recurring cross-domain reports of ``universal''
dimensions near this value (Figure~\ref{fig:gate}b).
\end{corollary}

\begin{corollary}[Lyapunov horizon]
Contraction accelerates \emph{state estimation}, not forecasting beyond the horizon: for a
system with maximal Lyapunov exponent $\lambda>0$, prediction is limited to
$t^\ast\approx\lambda^{-1}\log(1/\varepsilon)$ regardless of method. The observer wins the
transient, not the horizon.
\end{corollary}

\section{Empirical verification}
\label{sec:experiments}

All experiments use double precision, fixed seeds, and are re-run end to end by the
verification suite (\S\ref{sec:repro}).

\textbf{E1: recurrence law (Fig.~\ref{fig:rec}).} On a $1.2\times10^5$-step Lorenz
trajectory we measure mean first-return times to $r$-balls at 60 reference points. The
small-$r$ branch follows a power law with measured exponent $-1.79$ for this reference
sample ($-1.96$ for another; the spread across reference samples is of order $0.2$),
against the theoretical $-D_2\approx-2.05$; the residual bias is a finite-trajectory and
sampling-regime effect (see \S\ref{sec:limits}). Large $r$ saturates at the orbital period, as expected.

\begin{figure}[t]\centering
\includegraphics[width=0.62\linewidth]{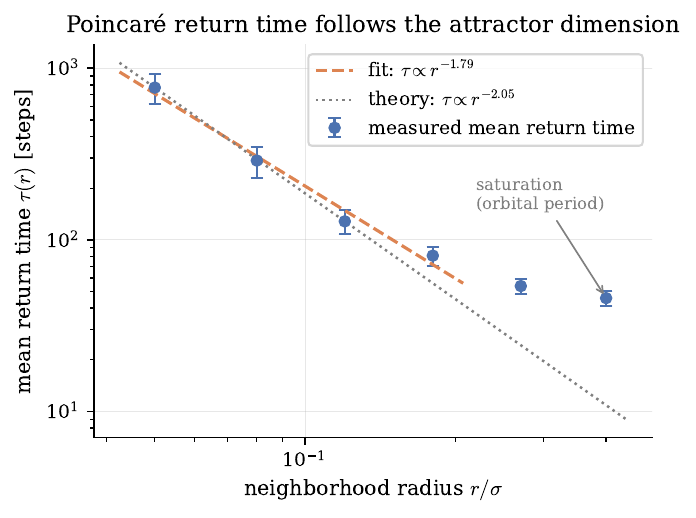}
\caption{Mean Poincar\'e return time on the Lorenz attractor versus neighborhood radius.
The asymptotic branch follows $\tau\propto r^{-D}$ with $D$ the dimension of the invariant
measure (Kac; Boshernitzan; Barreira--Saussol); the large-$r$ end saturates at the orbital
period. Error bars: standard error over 60 reference points.}
\label{fig:rec}
\end{figure}

\textbf{E2: contraction law (Fig.~\ref{fig:con}).} For random detectable systems
($n{=}4$) we iterate the Riccati recursion and record the step count $N_a(\varepsilon)$ at
which the gain reaches relative error $\varepsilon$. $N_a$ is linear in
$\log(1/\varepsilon)$ at fixed $\rho$ ($R^2=1.000$) and linear in $(1-\rho^2)^{-1}$ across
18 systems with $\rho\in[0.5,0.98]$ ($R^2=0.985$).

\begin{figure}[t]\centering
\includegraphics[width=0.95\linewidth]{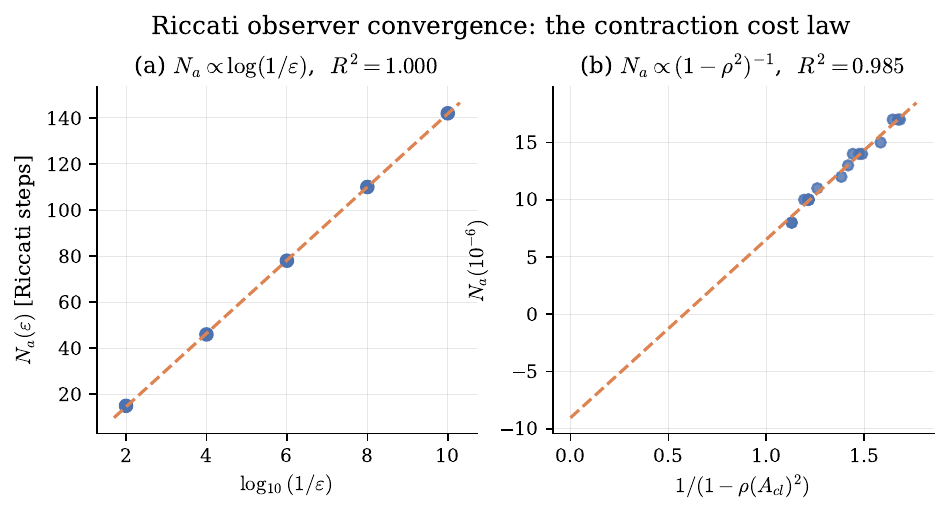}
\caption{Observer cost law. (a) Steps to relative gain error $\varepsilon$ versus
$\log_{10}(1/\varepsilon)$ for a fixed system. (b) Steps to $\varepsilon=10^{-6}$ versus
$(1-\rho(A_{cl})^2)^{-1}$ across random detectable systems.}
\label{fig:con}
\end{figure}

\textbf{E3: the separation (Fig.~\ref{fig:sep}).} Placing both measured cost curves on a
single tolerance axis realizes the theorem: at $\varepsilon=10^{-6}$ and $d\approx2$ the
measured gap is $\sim10^{9}$. We emphasize the honest caveat that $\varepsilon$ denotes a
matching radius for recurrence and a relative gain error for the observer; the theorem
compares the two mechanisms' native tolerances.

\begin{figure}[t]\centering
\includegraphics[width=0.62\linewidth]{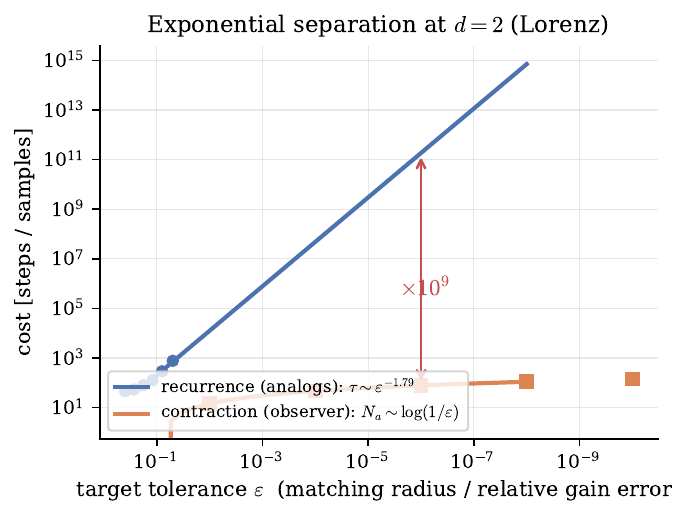}
\caption{Cost to reach tolerance $\varepsilon$: power law (recurrence, measured on Lorenz)
versus logarithm (contraction, measured Riccati). Points: measurements; lines: fitted
laws. The double arrow marks the measured $\sim10^{9}$ gap at $\varepsilon=10^{-6}$.}
\label{fig:sep}
\end{figure}

\textbf{E4: gate and estimator artifact (Fig.~\ref{fig:gate}).} The surrogate prediction
discriminant separates deterministic cores (Lorenz $19.9\times$, R\"ossler $15.2\times$)
from stochastic signals (white and $1/f^2$ noise, ${\approx}1\times$), including the
Osborne--Provenzale trap where $1/f^2$ noise shows a finite $\hat D\approx1.4$ with no
determinism. Panel (b): a fixed $m{=}16$ estimation pipeline reports
$\hat D\approx 8$ on pure white noise, stably across $N=10^3\ldots1.6\times10^4$, while
Lorenz reads its own $\hat D\approx1.7$---``universal'' dimensions track the pipeline, not
the world. Across five random seeds the gate ratios are $17.3\pm1.1$ (Lorenz),
$14.4\pm0.8$ (R\"ossler), $1.0\pm0.1$ (white noise), and $1.1\pm0.1$ ($1/f^2$ noise);
Figure~\ref{fig:gate}a shows one representative seed.

\textbf{E5: real data.} We ran the Kac--Riccati gate unchanged on two classical records. The Santa Fe
laser benchmark \citep{weigend1994}, a far-infrared laser in a chaotic state
($N{=}10{,}093$), is \emph{admitted}: surrogate ratio $3.9\pm0.7$ over five seeds, with
$\hat D_2=2.01$, well below the ceiling $2\log_{10}N=8$ and consistent with its known
Lorenz-like dynamics. The monthly sunspot series ($N{=}3{,}330$, SIDC) is
\emph{refused}: surrogate ratio $1.2\pm0.0$, even though a naive estimate returns a
finite $\hat D_2=2.26$---precisely the configuration in which 1980s--90s claims of a
low-dimensional solar attractor arose and were later retracted. One admission and one
refusal, both matching the settled literature, is the intended behavior of the protocol.

\begin{figure}[t]\centering
\includegraphics[width=0.95\linewidth]{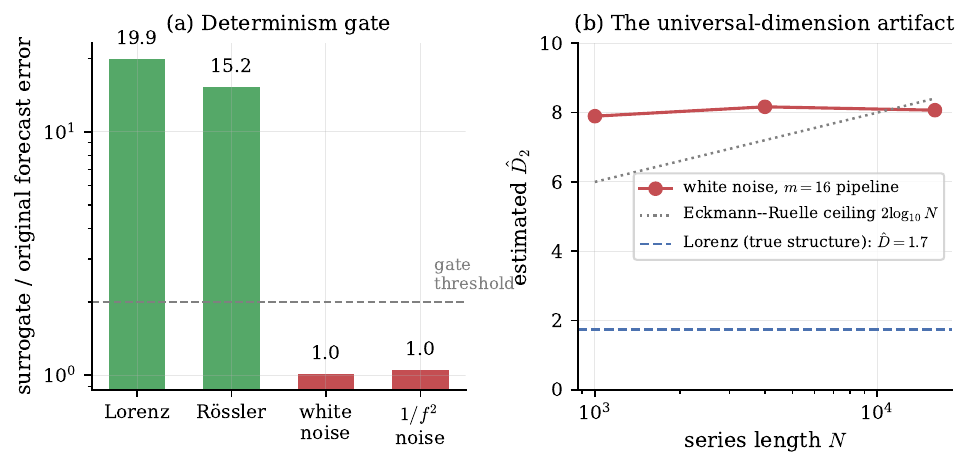}
\caption{(a) Admission gate: ratio of analog-forecast error on phase-randomized surrogates
to the original signal (log scale); the dashed line is the admission threshold. (b) A
fixed estimation pipeline prints $\hat D\approx8$ on structureless noise regardless of
$N$; the Eckmann--Ruelle ceiling $2\log_{10}N$ bounds believable estimates.}
\label{fig:gate}
\end{figure}

\section{The Kac--Riccati gate: an admission protocol}
\label{sec:gate}

\begin{algorithm}[h]
\caption{The Kac--Riccati gate (admission to Theorem~\ref{thm:main})}
\begin{algorithmic}[1]
\State \textbf{Input:} scalar series $y_{1:N}$
\State $r \gets \mathrm{err}_{\text{analog}}(\text{surrogate}(y))\,/\,\mathrm{err}_{\text{analog}}(y)$
  \Comment{phase-randomized surrogates \citep{theiler1992}}
\If{$r \le 2$} \State \Return \textsc{outside class} (stochastic / linear); do not interpret $\hat D$ \EndIf
\State estimate $\hat d$ by Grassberger--Procaccia with window $(m{-}1)\tau\approx1.5\,t_{\mathrm{acf}}$,
Theiler exclusion; \textbf{require} stability of $\hat d$ under $10\times N$ and
$\hat d < 2\log_{10}N$ \Comment{Cor.~\ref{cor:er}}
\State fit local $(A,C)$ on the reconstruction; run Riccati to $K_{ss}$;
budget $N_a=\log(1/\varepsilon)/(1-\rho^2)$
\State \Return \textsc{inside class}, dimension $\hat d$, expected gain
$\sim \varepsilon^{-\hat d}/N_a$
\end{algorithmic}
\end{algorithm}

The Kac--Riccati gate is deliberately conservative: a finite $\hat D$ alone never admits a signal
(that is precisely the Osborne--Provenzale failure), and an unstable $\hat D$ under data
growth is rejected as a ceiling artifact.

\section{A computational instance: anchor-subspace acceleration}
\label{sec:matmul}

Iterates $y_{k+1}=Ay_k$ of a fixed rank-$r$ operator $A=USV^\top$ satisfy
$y_k\in\mathrm{col}(U)$ for all $k\ge1$: the ``attractor'' is a linear subspace and its
dimension is the rank. The observer here is trivial---an orthonormal basis $Q$ of the
subspace identified from the first $r{+}2$ observed iterates (the anchors)---after which
each step costs $O(nr)$ instead of $O(n^2)$. At $n{=}1024$, $r{=}8$, $K{=}500$ steps this
yields a $19.4\times$ end-to-end FLOP reduction at maximum relative error
$5.2\times10^{-14}$: the same computation, faster, not an approximation. The boundary of
Remark~\ref{rem:bound} is equally concrete: for full-rank Gaussian $A$ the residual of the
iterates outside any rank-8 subspace is $0.97$, and the method correctly refuses to apply.
This is the mechanism underlying our broader program of anchor-based accelerators; the
present paper isolates the provable core.

\paragraph{From binary to stratified admission.}
The rank dichotomy of Remark~\ref{rem:bound} is the simplest instance of a general
picture: computational structure lives on \emph{degeneracy varieties}. Matrices of rank
$\le r$ form determinantal varieties in the projective space of operators, and the
achievable error of serving a nearby operator from the stratum equals its distance to the
variety: probing $A$ with a randomized range sketch \citep{halko2011} and projecting onto
the detected stratum, we measure $\log(\mathrm{err})=0.99\,\log(\mathrm{dist})$ over six
decades of perturbation---the Eckart--Young theorem \citep{eckart1936} acting as the
demon's accuracy certificate. Two strata must be distinguished in practice: the
\emph{trajectory} stratum (iterates collapse onto the dominant invariant subspace and can
be lower-dimensional than the operator's rank, permitting even larger legitimate skips)
and the \emph{operator} stratum (fresh inputs, governed by rank). Beyond matrices, the
stratification of degeneracy loci of polylinear forms is precisely the subject of
non-linear algebra \citep{dolotin1998,dolotin2007}, with boundary-format discriminant
algorithms as the tractable island and the NP-completeness of tensor rank
\citep{hastad1990} as the named wall. Extending stratified admission to tensor
contractions along these lines is the natural continuation of the present program.

\section{Limitations}
\label{sec:limits}

(1)~Part (ii) is proved for linear/linearizable observation models; for strongly nonlinear
reconstructions the observer guarantees are local. (2)~The separation concerns state
estimation cost; chaotic forecast horizons remain Lyapunov-limited for all methods.
(3)~Return-time exponents measured on finite trajectories are biased toward shallow
values (our $-1.79$ vs.\ $-2.05$), and the sampling regime matters: the exponent $-d$
holds when the sampling step exceeds the ball-crossing time (map regime); when several
samples fall inside one crossing, the flow regime yields $-(d-1)$. Strongly
phase-coherent flows compound this: on the R\"ossler attractor we measure $-0.71$ at the
native sampling and $-1.28$ after $4\times$ striding---moving toward but not reaching
$-d$, because quasi-periodic re-entries at the orbital period compress the scaling range.
We report these numbers rather than tune them away; the Lorenz measurements operate in
the map regime. (4)~The gate's
prediction discriminant weakens for strongly periodic signals, where surrogates retain
most predictability; we use the discriminant of \citet{sugihara1990} rather than the
$\hat D$ jump for this reason. (5)~The empirical scope is numerical; no claims are made
about physical measurement processes here.

\section{Reproducibility}
\label{sec:repro}

A single script, \texttt{demon\_observer\_verify.py}, re-derives every number in this
paper as a pass/fail assertion (17/17 at submission): the matmul instance (accuracy,
speedup, full-rank refusal), the recurrence exponent, both Riccati laws, the gate on four
signal classes, and the dimension-artifact panel. Figures regenerate from
\texttt{make\_figs.py}; the real-data experiment (E5) is scripted in
\texttt{realdata\_gate.py}, which fetches both public datasets. Both require only NumPy/SciPy/Matplotlib and finish in minutes on
a laptop.

\appendix
\section{A rigorous sample lower bound for analog prediction}
\label{app:lb}

\begin{proposition}\label{prop:lb}
Let $(X,f,\mu)$ be measure-preserving and let $x$ be a point at which the pointwise
dimension exists and equals $d$. Fix $\delta>0$. Then there exists
$\varepsilon_0(x,\delta)>0$ such that for all $\varepsilon<\varepsilon_0$ and all
$N\in\mathbb{N}$, a stationary record $y,f(y),\dots,f^{N-1}(y)$ with $y\sim\mu$ contains
an $\varepsilon$-analog of $x$ with probability at most $N\,\varepsilon^{\,d-\delta}$.
Consequently, achieving success probability at least $p$ requires
$N\ge p\,\varepsilon^{-(d-\delta)}$.
\end{proposition}

\begin{proof}
By the definition of pointwise dimension,
$\lim_{\varepsilon\to0}\log\mu(B_\varepsilon(x))/\log\varepsilon=d$, hence there is
$\varepsilon_0>0$ with $\mu(B_\varepsilon(x))\le\varepsilon^{\,d-\delta}$ for all
$\varepsilon<\varepsilon_0$. By invariance of $\mu$,
$\Pr[f^k(y)\in B_\varepsilon(x)]=\mu(B_\varepsilon(x))$ for every $k\ge0$. The union
bound over $k=0,\dots,N-1$ gives
$\Pr[\exists\,k<N:\ f^k(y)\in B_\varepsilon(x)]\le N\,\mu(B_\varepsilon(x))
\le N\,\varepsilon^{\,d-\delta}$.
\end{proof}

\begin{remark}
Proposition~\ref{prop:lb} uses no mixing assumptions and complements the almost-sure
exponent identity of \citet{barreira2001}: the proposition is a finite-sample lower
bound, the identity is the exact asymptotic rate, and the Kac formula
$\mathbb{E}_{\mu|_A}[\tau_A]=1/\mu(A)$ supplies the matching expectation statement.
\end{remark}

\end{document}